\documentclass[a4paper, draft,reqno,11pt]{amsart}
\usepackage[english]{babel}
\usepackage{amsmath}
\usepackage{amssymb}
\usepackage{amscd}
\usepackage{amsthm}
\newtheorem{proposition}{Proposition}
\newtheorem{lemma}{Lemma}
\newtheorem{theorem}{Theorem}
\newtheorem{pr}{Question}

\theoremstyle{definition}
\newtheorem{example}{Example}
\theoremstyle{remark}
\newtheorem {remark}{Remark}

\DeclareMathOperator{\rk}{rk} \DeclareMathOperator{\dg}{deg}
\DeclareMathOperator{\spec}{Spec }  \DeclareMathOperator{\aut}{Aut }
\DeclareMathOperator{\td}{tr.deg.}  
  \DeclareMathOperator{\lie}{Lie }
\begin{document}
\date{}
\title{On restriction of roots on affine
$T$-varieties}
\author[Polina Kotenkova]{Polina Yu. Kotenkova}
\address{Department of Higher Algebra, Faculty of Mechanics and Mathematics, Lomonosov Moscow State University, Leninskie Gory~1, Moscow, 119991, Russia}
\email{kotpy@mail.ru} \maketitle

\begin{abstract}
Let $X$ be a normal affine algebraic variety with regular action
of a torus $\mathbb{T}$ and $T\subset\mathbb{T}$ be a subtorus. We
prove that each root of $X$ with respect to $T$ can be obtained by
restriction of some root of $X$ with respect to $\mathbb{T}$. This
allows to get an elementary proof of the description of roots of
the affine Cremona group. Several results on restriction of roots
in the case of subtorus action on an affine toric variety are
obtained.
\end{abstract}

\section*{Introduction }

Let $\mathbb{T}$ be an algebraic torus. An algebraic variety
endowed with an effective action of $\mathbb{T}$ is called a
{\itshape $\mathbb{T}$-variety}. Let~$X$ be an affine
$\mathbb{T}$-variety. The algebra $A=\mathbb{K}[X]$ of regular
functions on $X$ is graded by the lattice $M$ of characters of the
torus $\mathbb{T}$. A derivation $\partial$ on an algebra $A$ is
said to be {\itshape locally nilpotent} (LND) if for each $a\in A$
there exists $n\in\mathbb{N}$ such that $\partial^n(a)=0$. A
locally nilpotent derivation on $A$ is said to be {\itshape
homogeneous} if it respects the $M$-grading. A homogeneous LND
$\partial$ shifts the $M$-grading by some lattice vector $\dg
\partial$. This vector is named the {\itshape degree} of $\partial$. The degrees of homogeneous LNDs
are called {\itshape $\mathbb{T}$-roots} of the
$\mathbb{T}$-variety $X$. This definition imitates in some sense
the notion of a root from Lie Theory. Let~$\mathbb{G}_a$ be the
additive group of the ground field $\mathbb{K}$. It is well known
that LNDs on $A$ are in bijective correspondence with regular
actions of~$\mathbb{G}_a$ on $X$. Namely, for given LND $\partial$
on $A$ the map $\varphi_{\partial}: \mathbb{G}_a\times
A\rightarrow A$,
$\varphi_{\partial}(t,\partial)=\exp(t\partial)(f)$, defines a
$\mathbb{G}_a$-action, and any regular $\mathbb{G}_a$-action on
$X$ arises in this way. It is easy to see that an LND on $A$ is
homogeneous if and only if the corresponding $\mathbb{G}_a$-action
is normalized by the torus~$\mathbb{T}$ in the group~$\aut (X)$.
In these terms, a root is a character by which $\mathbb{T}$ acts
on $\mathbb{G}_a$.

Any derivation on $\mathbb{K}[X]$ extends  to a derivation on the
field of fractions  $\mathbb{K}(X)$ by the Leibniz rule. A
homogeneous LND $\partial$ on $\mathbb{K}[X]$ is said to be of
{\itshape fiber type} if $\partial(\mathbb{K}(X)^{\mathbb{T}})=0$
and of {\itshape horizontal type} otherwise. In other words,
$\partial$ is of fiber type if and only if the general orbits of
corresponding $\mathbb{G}_a$-action on $X$ are contained in the
closures of $\mathbb{T}$-orbits. Recall that for a
$\mathbb{T}$-action on an algebraic variety, the {\itshape
complexity} is defined as the codimension of a general orbit.
Toric varieties are $\mathbb{T}$-varieties of complexity zero. In
this case a description of roots was obtained by
M.~Demazure~\cite{De}, see also \cite[Section 4]{Co}. There is a
complete description of the homogeneous LNDs and roots in the case
of $\mathbb{T}$-actions of complexity at most one due to
A.~Liendo, see \cite{L1}. A classification of homogeneous LNDs of
fiber type in arbitrary complexity is given in \cite{L1'}. The
results of \cite{L1} and \cite{L1'} are based on the combinatorial
description of affine $\mathbb{T}$-varieties obtained by
K.~Altmann and J.~ Hausen, see \cite{AH}.

Let $T\subset\mathbb{T}$ be a subtorus. The torus $T$ also acts on
$X$. Denote by $M_{T}$ and $M_{\mathbb{T}}$ the character lattices
of $T$ and $\mathbb{T}$ respectively. It is obvious that if an LND
respects the $M_{\mathbb{T}}$-grading on $\mathbb{K}[X]$, then it
respects the $M_{T}$-grading as well. So any $\mathbb{T}$-root of
a variety $X$ can be restricted to some $\,T$-root. In this paper
we show that any $\,T$-root of $X$ arises in this way (Theorem
\ref{mt}). Also we study the restriction of roots in the case of
affine toric varieties. For affine toric surfaces a complete
description of this restriction is obtained.

In Sections~\ref{s1} we recall some basic facts on locally
nilpotent derivations and roots. Section~\ref{s2} contains an
elementary proof of surjectivity of restriction of roots.  In
Section~\ref{s3} it is shown how the description of roots of the
affine Cremona group given by A.~Liendo in \cite{L2} can be
obtained by our method. Note that Liendo's proof is based on the
classification of LNDs of horizontal type in the case of
$\mathbb{T}$-actions of complexity one. We give more direct proof.
Section~\ref{s4} is devoted to the study of the restriction of
roots on an affine toric variety. In Section~\ref{s5} the case of
affine toric surfaces is considered in more detail.

We work over an algebraically closed field $\mathbb{K}$ of
characteristic zero.

\section{\label{s1} Homogeneous locally nilpotent derivations}

Let $X$ be an affine variety with an effective action of an
algebraic torus $\mathbb{T}$ and $A=\mathbb{K}[X]$ be the algebra
of regular functions on $X$. Denote by $M$ the character lattice
of the torus~$\mathbb{T}$ and by $N$ the lattice of one-parameter
subgroups of $\mathbb{T}$. It is well known that there is a
bijective correspondence between effective $\mathbb{T}$-actions on
$X$ and effective $M$-gradings on~$A$. In fact, the algebra $A$ is
graded by a semigroup of lattice points in some convex polyhedral
cone $\omega\subseteq
M_{\mathbb{Q}}=M\otimes_{\mathbb{Z}}\mathbb{Q}$ called the
{\itshape weight cone}. So we have $$A=\bigoplus_{m\in \omega_{
M}} A_m\chi^m,$$ where $\omega_{ M}=\omega\cap M$ and $\chi^m$ is
the character corresponding to $m$.

A derivation $\partial$ on an $M$-graded algebra $A$ is called
{\itshape homogeneous} if it sends homogeneous elements to
homogeneous ones.  If ${a,b\in A\backslash \ker \partial}$ are
some homogeneous elements, then
${\partial(ab)=a\partial(b)+\partial(a)b}$ is homogeneous too and
hence ${\dg
\partial(a)-\dg a=\dg
\partial(b)-\dg b}$. So any homogeneous
derivation $\partial$ has a well defined {\itshape degree} given
as $\dg\partial=\dg
\partial(a)-\dg a$ for any homogeneous $a\in A\backslash \ker
\partial$.
Homogeneous LNDs are called the {\itshape root vectors} with
respect to the torus $\mathbb{T}=\spec \mathbb{K}[M]$ and their
degrees are said to be the {\itshape $\mathbb{T}$-roots}.

The following well known lemma shows that any LND decomposes into
a sum of homogeneous derivations, some of which are locally
nilpotent, cf.~\cite[Lemma 1.10]{L1}.

\begin{lemma} \label{dec} Let $A$ be a finitely generated $M$-graded
domain. For any derivation $\partial$ on $A$ there is a
decomposition \begin{equation} \label{f2}
\partial=\sum\limits_{e\in M}
\partial_e,\end{equation} where $\partial_e$ is a homogeneous derivation of
degree $e$. Moreover, the convex hull $\Delta(\partial)\subset
M_{\mathbb{Q}}$ of the set ${\{e\in M\mid \partial_e\ne 0\}}$ is a
polytope and for every vertex $e$ of $\Delta(\partial)$, the
derivation $\partial_e$ is locally nilpotent if $\partial$ is.
\end{lemma}

\begin{proof} For every homogeneous $a\in A$ there is a
decomposition ${\partial (a)=\sum_{e\in M} a_{\dg (a)+e}}$, where
the element $a_{\dg (a)+e}$ is homogeneous of degree $\dg (a)+e$.
The linear map $\partial_e$ given by the rule
$\partial_e(a)=a_{\dg (a)+e}$ is a homogeneous derivation of
degree $e$, and ${\partial=\sum_{e\in M}
\partial_e}$. Since $A$ is finitely generated, the set $\{e\in M\mid \partial_e\ne
0\}$ is finite and its convex hull $\Delta(\partial)$ is a
polytope. Let $\widehat{e}$ be a vertex of $\Delta(\partial)$ and
$a\in A_m$ be some homogeneous element. Then $\partial^n(a)$ has
$\partial_{\widehat{e}}^n(a)$ as its homogeneous component of
degree $m+n\widehat{e}$. So if $\partial$ is LND, then
$\partial_{\widehat{e}}$ is LND as well.
\end{proof}

Let $X$ be an affine toric variety, i.~e. a normal affine variety
with a generically transitive action of a torus $\mathbb{T}$. In
this case
$$A=\bigoplus_{m\in \omega_{M}} \mathbb{K}\chi^m=\mathbb{K}[\omega_{M}]$$
is the semigroup algebra.  Recall that for given cone
$\omega\subset M_{\mathbb{Q}}$, its {\itshape dual cone} is
defined by
$$\sigma=\{n\in N_{\mathbb{Q}}\,|\,\langle
n,p\rangle\geqslant0\,\,\,\forall p\in\omega\},$$ where
$\langle,\rangle$ is the pairing between dual lattices $N$ and
$M$. Let $\sigma(1)$ be the set of rays of a cone $\sigma$ and
$n_{\rho}$ be the primitive lattice vector on the ray $\rho$. By
definition, for $\rho\in\sigma(1)$ set
$$S_{\rho}:=\{e\in M\,|\, \langle n_{\rho},e\rangle=-1
\,\,\mbox{and}\,\, \langle n_{\rho'},e\rangle\geqslant0
\,\,\,\,\forall\,\rho'\in \sigma(1), \,\rho'\ne\rho\}.$$ One
easely checks that the set $S_{\rho}$ is infinite for any
$\rho\in\sigma(1)$. The elements of the set
$\mathfrak{R}:=\bigsqcup\limits_{\rho} S_{\rho}$ are called the
{\itshape Demazure roots} of $\sigma$. This notion was introduced
in \cite{De}. Let $e\in S_{\rho}$. One can define the homogeneous
LND on the algebra $A$ by the rule
$$\partial_e(\chi^m)=\langle n_{\rho},m\rangle\chi^{m+e}.$$ It
turns out that the set $\{\alpha\partial_e\,|\, \alpha\in
\mathbb{K}, e\in \mathfrak{R}\}$ coincides with the set of all
homogeneous LNDs on~$A$, see \cite[Theorem 2.7]{L1}.

\begin{example} \label{e1} Consider $X=\mathbb{A}^d$ with the standard action of the torus
$(\mathbb{K}^{\times})^d$. It is a toric variety with the cone
$\sigma=\mathbb{Q}^d_{\geqslant0}$ having rays
$\rho_1=\langle(1,0,\ldots,0)\rangle_{\mathbb{Q}_{\geqslant0}},\ldots,\rho_d=\langle(0,0,\ldots,0,1)\rangle_{\mathbb{Q}_{\geqslant0}}$.
The dual cone $\omega$ is $\mathbb{Q}^d_{\geqslant0}$ as well. In
this case
$$S_{\rho_i}=\{(e_1,\ldots,e_{i-1},-1,e_{i+1},\ldots,e_d)\,|\,e_j\in\mathbb{Z}_{\geqslant0}\}.
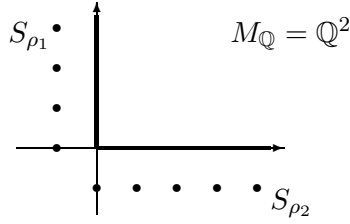
$$
\vspace{0.05cm}
\begin{center}
\begin{picture}(100,75)
\multiput(50,15)(15,0){5}{\circle*{3}}
\multiput(35,30)(0,15){4}{\circle*{3}}
\put(20,30){\vector(1,0){100}} \put(50,5){\vector(0,1){80}}
\put(17,70){$S_{\rho_1}$} \put(115,7){$S_{\rho_2}$}
\put(100,70){$M_{\mathbb{Q}}=\mathbb{Q}^2$} \linethickness{0.5mm}
\put(50,30){\line(1,0){65}} \put(50,30){\line(0,1){50}}
\end{picture}
\end{center}
Denote
$x_1=\chi^{(1,0,\ldots,0)},\ldots,x_d=\chi^{(0,\ldots,0,1)}$. Then
$\mathbb{K}[X]=\mathbb{K}[x_1,\ldots,x_d]$. It is easy to see that
the homogeneous LND corresponding to the root
$e=(e_1,\ldots,e_d)\in S_{\rho_i}$ is
$$\partial_e=x_1^{e_1}\ldots x_{i-1}^{e_{i-1}} x_{i+1}^{e_{i+1}}\ldots x_{d}^{e_{d}}\frac{\partial}{\partial x_i}.$$
\end{example}

\section{\label{s2} Restriction of roots}

Consider an affine $\mathbb{T}$-variety $X$ and a subtorus
$T\subset\mathbb{T}$. The torus $T$ also acts on $X$. Evidently,
every $\mathbb{T}$-homogeneous LND on $\mathbb{K}[X]$ is
$T$-homogeneous and so any $\mathbb{T}$-root of $X$ can be
restricted to some $T$-root. The following theorem proves that
every $T$-root arises in this way.

\begin{theorem} \label{mt} Let $X$ be an affine variety with a regular action of a torus
$\mathbb{T}$ and $T\subset\mathbb{T}$ be a subtorus. Then the
restriction of  $\,\mathbb{T}$-roots to  $T$-roots is surjective.
Moreover, if a $T$-root $e$ is the restriction of the only one
$\mathbb{T}$-root, then any $T$-homogeneous LND on $\mathbb{K}[X]$
of degree~$e$ is $\mathbb{T}$-homogeneous as well.
\end{theorem}

\begin{proof} Let $e$ be a $T$-root of $X$ and $\partial$ be a homogeneous LND of degree~$e$. By Lemma
\ref{dec}, $\partial$~decomposes into a sum of
$\mathbb{T}$-homogeneous derivations, some of which are locally
nilpotent. So degree of any locally nilpotent summand is the
$\mathbb{T}$-root of $X$ and its restriction to $T$ equals $e$. If
$e$ is the restriction of the only one $\mathbb{T}$-root, then
there is only one summand in the decomposition of $\partial$, and
hence $\partial$ is $\mathbb{T}$-homogeneous.
\end{proof}

Now we consider an affine toric variety $X$ with the acting torus
$\mathbb{T}$ and a subtorus $T\subset\mathbb{T}$ of codimension
one. Let $M$ be the character lattice of $\mathbb{T}$ and $N$ be
the lattice of  one-parameter subgroups of $\mathbb{T}$. Denote by
$\sigma_{X}\subset N_{\mathbb Q}$ the cone corresponding to $X$
and by $\Gamma_{T}\subset N_{\mathbb Q}$ the hyperplane generated
by one-parameter subgroups of $\mathbb{T}$ that are contained in
the subtorus $T$.

\begin{proposition} \label{sl} Let $X$ be an affine toric variety with an acting torus
$\mathbb{T}$ and $T\subset\mathbb{T}$ be a subtorus of codimension
one. If\,\, ${\Gamma_{T}\cap \sigma_{X}=\{0\}}$, then the
restriction of roots is bijective. In particular, there is only
one (up to scalar) root vector for each $T$-root $e$ and any
$T$-homogeneous LND on $\mathbb{K}[X]$ is $\mathbb{T}$-homogeneous
as well.
\end{proposition}

\begin{proof}Let $\langle \cdot, m_{T}\rangle=0$
be the equation of the  hyperplane $\Gamma_{T}$, where $m_{T}\in
M$. The vector $m_{T}$ generates the kernel of the restriction of
characters from $\mathbb{T}$ to $T$. Since ${\Gamma_{T}\cap
\sigma_{X}=\{0\}}$, we may assume that $\langle p,m_{T}\rangle>0$
for all $p\in \sigma_{X}\backslash\{0\}$. Suppose the restrictions
of the roots $e_1$ and $e_2$ from $\mathbb{T}$ to $T$ coincide and
$e_1-e_2=\lambda m_{T}$, where~$\lambda>0$. Let $\rho$ be the ray
of $\sigma_X$ with $e_1\in S_{\rho}$. Then $\langle
n_{\rho},e_1-e_2\rangle=-1-\langle n_{\rho},e_2\rangle\leqslant
0$. On the other hand, $\langle
n_{\rho},e_1-e_2\rangle=\lambda\langle n_{\rho},m_T\rangle>0$.
This contradiction implies that the restriction of roots from
$\mathbb{T}$ to $T$ is injective. By Theorem \ref{mt}, any
$T$-homogeneous LND on $\mathbb{K}[X]$ is
$\mathbb{T}$-homogeneous.
\end{proof}

\begin{remark}  In Proposition \ref{sl} we can not replace a subtorus of codimension one  with a
subtorus of arbitrary codimension. For example, let
$X=\mathbb{A}^3$ and $T=\{(t,t,t^{-1})\mid t\in
\mathbb{T}^{\times}\}$. Denote by $\Gamma_{T}$ the line
corresponding to $T$. Then $\Gamma_{T}=\mathbb{Q}\cdot(1,1,-1)$
and $\Gamma_{T}\cap \sigma_{X}=\{0\}$, but the restriction of
roots is not bijective. Indeed, the $T$-homogeneous LND
$\partial=x_2x_3\frac{\partial}{\partial
x_1}+\frac{\partial}{\partial x_2}$ is not
$\mathbb{T}$-homogeneous.
\end{remark}

\section{\label{s3} Roots of the affine Cremona group}

Consider the polynomial algebra
$\mathbb{K}^{[n]}=\mathbb{K}[x_1,\ldots,x_n]$. The set of "volume
preserving" transformations $$\aut^{*}_{\mathbb{K}}
\mathbb{K}^{[n]}=\{\gamma\in\aut_{\mathbb{K}}
\mathbb{K}^{[n]}\,|\, \det \left(\frac{\partial
\gamma(x_j)}{\partial x_i}\right)_{1\leqslant i,j\leqslant
n}=1\}$$ is a closed normal subgroup of the {\itshape affine
Cremona group} $\aut_{\mathbb{K}} \mathbb{K}^{[n]}$. It is an
infinite dimensional simple algebraic group, see \cite{Sh}. Every
maximal algebraic torus in $\aut^{*}_{\mathbb{K}}
\mathbb{K}^{[n]}$ has dimension $n-1$ and is conjugate to
$$T=\{\gamma\in\aut^{*}_{\mathbb{K}} \mathbb{K}^{[n]}\,|\, \gamma(x_i)=t_ix_i,\, t_i\in
\mathbb{K},\, \prod_{i=1}^{n}t_i=1\}.$$ If $\partial$ is an LND on
$\mathbb{K}^{[n]}$, then $\exp t\partial\in\aut^{*}_{\mathbb{K}}
\mathbb{K}^{[n]}$ for any $t\in \mathbb{K}$, and so
$\partial\in\lie(\aut^{*}_{\mathbb{K}} \mathbb{K}^{[n]})$. A
nonzero LND $\partial$ of $\mathbb{K}^{[n]}$ is called a {\itshape
root vector} of $\aut^{*}_{\mathbb{K}} \mathbb{K}^{[n]}$ with
respect to $T$ if there exists a nontrivial character $\chi:
T\rightarrow \mathbb{K}^{\times}$ such that
$$\gamma\circ\partial\circ\gamma^{-1}=\chi(\gamma)\partial
\,\,\,\,\,\forall\, \gamma\in T.$$ The character $\chi$ is said to
be the {\itshape root} of $\aut^{*}_{\mathbb{K}} \mathbb{K}^{[n]}$
with respect to $T$ corresponding to $\partial$. It is easy to see
that two last definitions are equivalent to the definitions of
root vectors and roots given in Section~\ref{s2} in the case
$X=\mathbb{A}^n$. The first Popov's question in \cite{Po} was to
find all roots and root vectors of $\aut^{*}_{\mathbb{K}}
\mathbb{K}^{[n]}$ with respect to~$T$. The answer is given in
\cite{L2}. It claims that the root vectors are exactly the LNDs
$x^{\alpha}\frac{\partial}{\partial x_i}$ , where $x^{\alpha}$ is
any monomial not depending on $x_i$. Let us show that this
statement immediately follows from Proposition \ref{sl}. We have
an affine toric variety~$\mathbb{A}^n$ with the standard action of
the torus $\mathbb{T}=(\mathbb{K}^{\times})^n$ and the subtorus
$T\subset\mathbb{T}$ given by equation
$\prod\limits_{i=1}^{n}t_i=1$. The cone corresponding to
$\mathbb{A}^n$ is just $\sigma_{X}=\mathbb{Q}^n_{\geqslant0}$. The
hyperplane $\Gamma_{T}$ is given by equation $z_1+\ldots+z_n=0$
and we get $\Gamma_{T}\cap \sigma_{X}=\{0\}$. Thus all root
vectors are homogeneous with respect to $\mathbb{T}$ and hence, as
it is shown in Example~\ref{e1}, have the form $\lambda\,
x^{\alpha}\frac{\partial}{\partial x_i}$, where $\lambda\in
\mathbb{K}^{\times}$. So we obtain the following theorem.

\begin{theorem} The set of root vectors of $\aut^{*}_{\mathbb{K}} \mathbb{K}^{[n]}$ with respect to
$T$ coincides with the set $${\{\lambda\,
x^{\alpha}\frac{\partial}{\partial x_i}\,\mid\,\lambda\in
\mathbb{K}^{\times}, i\in\{1,\ldots,n\},\alpha\in
\mathbb{Z}^n_{\geqslant 0}, \alpha_i=0 \}}.$$ The root
corresponding to $\lambda\, x^{\alpha}\frac{\partial}{\partial
x_i}$ is the character $\chi_{i,\alpha}: T\rightarrow
\mathbb{K}^{\times}$ given by
$\chi_{i,\alpha}(\gamma)=t_i^{-1}\prod\limits_{j=1}^{n}t_j^{\alpha_j}.$
\end{theorem}

\section{\label{s4} Subtorus actions on affine toric varieties}

Let us start with three questions concerning roots on affine
$T$-varieties. The first one was communicated to us by A.~Liendo.

\begin{pr} Let $e$ be a root of an affine $T$-variety. How many root vectors
correspond to the root~$e$?
\end{pr}

As we have seen, in the toric case there is only one (up to
scalar) root vector for each root. It follows from \cite[Theorem
2.4]{L1'} that root vectors of fiber type having degree $e$ form
the vector space (of finite or infinite dimension). For root
vectors of horizontal type there is no complete answer even for
$T$-action of complexity one.

\begin{pr} Let $X$ be an affine $\mathbb{T}$-variety
 and $T\subset\mathbb{T}$ be a subtorus. Consider a
$T$-root $e$ of $X$. How many $\mathbb{T}$-roots  have $e$ as
their restrictions to $T$?
\end{pr}

\begin{pr} \label{qu3} Are all $T$-homogeneous LNDs of degree $e$ on $\mathbb{K}[X]$ homogeneous with respect
to $\mathbb{T}$?
\end{pr}

Theorem \ref{mt} shows that if a $T$-root $e$ is the restriction
of only one $\mathbb{T}$-root $\widehat{e}$, then any
$T$-homogeneous LNDs of degree $e$ is $\mathbb{T}$-homogeneous and
there are as many root vectors corresponding to $e$ as root
vectors corresponding to $\widehat{e}$.

Let us investigate these questions in the toric case. Consider an
affine toric variety $X$ with the acting torus $\mathbb{T}$ and a
subtorus $T\subset\mathbb{T}$ of codimension one. As above, denote
by $N$ the lattice of  one-parameter subgroups, by $M$  the
character lattice of $\mathbb{T}$, by $\sigma_{X}\subset
N_{\mathbb{Q}}$ the cone corresponding to $X$, and by $\Gamma_{T}$
the hyperplane corresponding to subtorus~$T$. Let $\langle \cdot,
m_{T}\rangle=0$ be the equation of the hyperplane $\Gamma_{T}$,
where $m_{T}\in M$. The set of $\mathbb{T}$-roots of the variety
$X$ is the set
$\frak{R}_{\mathbb{T}}=\coprod\limits_{\rho\in\sigma_{X}(1)}
S_{\rho}$ of Demazure roots of the cone $\sigma_{X}$. Let
$\frak{R}_{T}$ be the set of $T$-roots and $\pi:
\frak{R}_{\mathbb{T}} \rightarrow \frak{R}_{T}$ be the restriction
of roots.

The answers to Questions $1$-$3$ depend on relative position of
the cone $\sigma_{X}$ and the hyperplane~$\Gamma_{T}$. We already
have studied the case $\sigma_{X}\cap\Gamma_{T}=\{0\}$. By
Proposition \ref{sl} there is only one root vector for every
$T$-root $e$, only one $\mathbb{T}$-root has $e$ as its
restriction, and any $T$-homogeneous LND is
$\mathbb{T}$-homogeneous as well. The following three propositions
describe the restriction of roots in other cases.

\vspace{0.1cm}

\begin{proposition} \label{sl''} If $\Gamma_{T}$ intersects the interior of the cone $\sigma_{X}$
and does not contain the rays of $\sigma_{X}$, then any $T$-root
is the restriction of at most two $\mathbb{T}$-roots.
\end{proposition}

\begin{proof}
Suppose that the restrictions of $\mathbb{T}$-roots
$\widehat{e}_1\in S_{\rho_1},\widehat{e}_2\in S_{\rho_2}$ and
$\widehat{e}_3\in S_{\rho_3}$ coincide. If~${\rho_i=\rho_j}$, then
$\langle n_{\rho_i},e_i-e_j\rangle=0$ and the ray $\rho_i$ is
contained in $\Gamma_{T}$. So $\rho_i\ne\rho_j$ whenever $i\ne j$.
We may assume without loss of generality that
$\widehat{e}_1-\widehat{e}_2=\lambda m_{T}$ and
$\widehat{e}_2-\widehat{e}_3=\mu m_{T}$, where $\lambda>0$ and
$\mu>0$. Then \,\,${\lambda\langle
n_{\rho_2},m_{T}\rangle=1+\langle
n_{\rho_2},\widehat{e}_1\rangle>0}$\,\, and \,\,${\mu\langle
n_{\rho_2},m_{T}\rangle=-1-\langle
n_{\rho_2},\widehat{e}_3\rangle<0}$. So we get a contradiction.
\end{proof}

\begin{proposition}  \label{sl'}   Suppose that ${\Gamma_{T}\cap \sigma_{X}}$ is a face of the cone
$\sigma_{X}$ of positive dimension. Then

\begin{itemize}

\item[$a)$] ${\pi(S_{\rho_1})\cap \pi(S_{\rho_2})=\varnothing}$
whenever $\rho_1\ne\rho_2$.

\item[$b)$] If a ray $\rho$ is not contained in ${\Gamma_{T}\cap
\sigma_{X}}$, then $\left.\pi\right|_{S_{\rho}}:
S_{\rho}\rightarrow \pi(S_{\rho})$ is bijective.

\item[$c)$] If $\rho\subseteq{\Gamma_{T}\cap \sigma_{X}}$, then
for any $e\in\pi(S_{\rho})$ there are infinitely many elements of
$S_{\rho}$, having $e$ as their restriction, and root vectors
corresponding to $e$ form an infinite dimensional vector space.
\end{itemize}
\end{proposition}

\begin{proof} $a), b)$ If ${\Gamma_{T}\cap \sigma_{X}}$ is a face of the cone
$\sigma_{X}$, then $\langle p,m_{T}\rangle\geqslant0$ for all
$p\in \sigma_{X}$. Suppose $\widehat{e}_1\in S_{\rho_1}$,
$\widehat{e}_2\in S_{\rho_2}$, and
$\widehat{e}_1-\widehat{e}_2=\lambda m_{T}$, where $\lambda>0$.
Then $\lambda\langle n_{\rho_1},m_{T}\rangle=-1-\langle
n_{\rho_1},\widehat{e}_2\rangle\leqslant0$ and ${\lambda\langle
n_{\rho_2},m_{T}\rangle=1+\langle
n_{\rho_2},\widehat{e}_1\rangle\geqslant0}$. It is possible if and
only if $\rho_1=\rho_2\subseteq{\Gamma_{T}\cap \sigma_{X}}$.

$c)$ Let $\widehat{e}\in S_{\rho}$, where
$\rho\subseteq{\Gamma_{T}\cap \sigma_{X}}$. The vector
$\widehat{e}+\lambda m_{T}$ is a root if and only if it belongs to
the lattice and ${\langle
n_{\rho'},\widehat{e}\rangle+\lambda\langle
n_{\rho'},m_{T}\rangle\geqslant 0}$ for each $\rho'\ne\rho$. So
there are infinite many roots, whose restrictions are equal to
restriction of $\widehat{e}$. It remains to note that if
$\,\widehat{e}_1, \widehat{e}_2\in S_{\rho}$, then corresponding
LNDs commute and hence their sum is again LND.
\end{proof}

\begin{proposition} Suppose $\Gamma_{T}$ intersects the interior of the cone
$\sigma_{X}$ and contains some rays of $\sigma_{X}$. Then
\begin{itemize}

\item[$a)$] If $\rho\not\subset{\Gamma_{T}\cap \sigma_{X}}$, then
any $e\in\pi(S_{\rho})$ is the restriction of at most two
$\mathbb{T}$-roots.

\item[$b)$] If $\rho\subset{\Gamma_{T}\cap \sigma_{X}}$, then any
$e\in\pi(S_{\rho})$ is the restriction of a finite number $k_e$ of
$\mathbb{T}$-roots, there exists $e$ with $k_e\geqslant2$, and
root vectors corresponding to $e$ form a $k_e$-dimensional vector
space.

\end{itemize}
\end{proposition}

\begin{proof}
$a)$ Suppose that $e=\pi(\widehat{e}_1)=\pi(\widehat{e}_2)$, where
$\widehat{e}_1\in S_{\rho}$, $\widehat{e}_2\in S_{\rho'}$,
$\rho\not\subset{\Gamma_{T}\cap \sigma_{X}}$, and
$\rho'\subset{\Gamma_{T}\cap \sigma_{X}}$. Then
${\widehat{e}_1-\widehat{e}_2=\lambda m_{T}}$ and hence $\langle
n_{\rho'},\widehat{e}_1\rangle=-1$. It is a contradiction. If
${e\in\pi(S_{\rho_1})\cap\pi(S_{\rho_2})\cap\pi(S_{\rho_3})}$,
where $\rho_i\not\subset{\Gamma_{T}\cap \sigma_{X}}$, we can use
the same arguments as in Proposition \ref{sl''}.

$b)$ Let us show that any $e=\pi(\widehat{e})$, where
$\widehat{e}\in S_{\rho}$ and $\rho\subset{\Gamma_{T}\cap
\sigma_{X}}$, is the restriction of a finite number of elements.
As above, the vector $\widehat{e}+\lambda m_{T}$ is a root if and
only if it belongs to the lattice $N$ and $\langle
n_{\rho'},\widehat{e}\rangle+\lambda\langle
n_{\rho'},m_{T}\rangle\geqslant 0$ for any $\rho'\ne\rho$. In this
case there are rays $\rho'$ with both positive and negative values
$\langle n_{\rho'},m_{T}\rangle$. Hence only a finite number of
$\lambda$ satisfies these conditions. In conclusion, let
$\widehat{e}\in S_{\rho}$ be a root such that ${\langle
n_{\rho'},\widehat{e}\rangle}\geqslant-{\langle
n_{\rho'},m_{T}\rangle}$ whenever $\langle
n_{\rho'},m_{T}\rangle<0$. Then $\widehat{e}+ m_{T}$ is
$\mathbb{T}$-root and $k_e\geqslant2$ for $e=\pi(\widehat{e})$.
\end{proof}

\begin{example} Consider $X=\mathbb{A}^3$ with the standard action of the torus
$\mathbb{T}=(\mathbb{K}^{\times})^3$. Let
${T=\{(s_1,s_1,s_2)\,\mid\, s_1, s_2\in \mathbb{K}^{\times}\}}$ be
a subtorus of codimension one. In this case the
hyperplane~$\Gamma_{T}$ is generated by vectors $(1,1,0)$ and
$(0,0,1)$. It can be easily checked that
$$\mathfrak{R}_T=\{(a,b), (c,-1)\,\,\mid\,\,
a\in\mathbb{Z}_{\geqslant-1},\, b,c\in
\mathbb{Z}_{\geqslant0}\},$$ each root $(a,b)$ is the restriction
of two $\mathbb{T}$-roots, and each root $(c,-1)$ is the
restriction of $c+1$ $\mathbb{T}$-roots.
\end{example}

Thus we have a complete description of the restriction of roots in
the case, when ${\sigma_{X}\cap\Gamma_{T}}$ is a face of the
cone~$\sigma_{X}$. If the hyperplane $\Gamma_{T}$ intersect the
interior of $\sigma_{X}$, the answers to Questions $1$ and $3$ are
remain unknown for $T$-roots
$e\in\pi(S_{\rho_1})\cap\pi(S_{\rho_2})$, where
$\rho_1,\rho_2\not\subset \Gamma_{T}$. The following proposition
gives some sufficient condition on $\Gamma_{T}$, under which there
are infinitely many root vectors corresponding to $e$. In
Section~\ref{s5} we show that this condition is also necessary in
the case of affine toric surfaces.

\begin{proposition} \label{nhrv} Suppose
that  $e_1\in S_{\rho_1}$, $e_2\in S_{\rho_2}$, $\rho_1\ne\rho_2$,
$\pi(e_1)=\pi(e_2)$, and $\langle n_{\rho_1}, m_{T}\rangle=-1$.
Then the linear map ${\partial:\mathbb{K}[X]\rightarrow
\mathbb{K}[X]}$ given by
\begin{equation}\partial\chi^m=\chi^{m+e_2}(\alpha\langle n_{\rho_1},m\rangle\chi^{m_{T}}+\beta\langle n_{\rho_2},m\rangle)
(\alpha\chi^{m_{T}}-\beta\langle n_{\rho_2},m_{T}\rangle)^{\langle
n_{\rho_1},e_2\rangle},\,\,\,\mbox{for
all}\,\,\,m\in\omega_{M},\end{equation} is a $T$-homogeneous LND
of degree $\pi(e_1)$ for every $\alpha,\beta\in \mathbb{K}$.
\end{proposition}
\begin{proof} By assumptions, ${e_1-e_2=(\langle n_{\rho_1},e_2\rangle+1)m_{T}}$ and $\partial$ is a $T$-homogeneous derivation on $A$. By direct calculation we
get
$$\partial^{k+1}\chi^m=\chi^{m+(k+1)e_2}(\alpha\chi^{m_{T}}-\beta\langle n_{\rho_2},m_{T}\rangle)^{(k+1)\langle n_{\rho_1},e_2\rangle}\sum_{j=0}^{k+1}
\alpha^j\beta^{k-j+1}l_j^{(k+1)}\chi^{jm_{T}},$$ where
${l_j^{(k+1)}=(\langle
n_{\rho_1},m\rangle-j+1)l_{j-1}^{(k)}+(\langle
n_{\rho_2},m\rangle+j\langle
n_{\rho_2},m_{T}\rangle-k)l_j^{(k)}}$, ${l_1^{(1)}=\langle
n_{\rho_1},m\rangle}$, and ${l_0^{(1)}=\langle
n_{\rho_2},m\rangle}$. Note that $l_j^{(k)}=0$ whenever
$k\geqslant\langle n_{\rho_2},m\rangle+d\langle n_{\rho_2},
m_{T}\rangle+1$ and $j\leqslant d$, and $l_i^{(k)}=0$ whenever
$k\geqslant\langle n_{\rho_1},m\rangle+1$ and $i\geqslant\langle
n_{\rho_1},m\rangle+1$. Thus if $k=\langle
n_{\rho_2},m\rangle+\langle n_{\rho_2},m_{T}\rangle\langle
n_{\rho_1},m\rangle$, then $\partial^{k+1}\chi^m=0$, and
$\partial$ is locally nilpotent.
\end{proof}

\section{\label{s5} Affine toric surfaces}

Normal affine surfaces with $\mathbb{C}^{\times}$-actions and LNDs
on them were studied in \cite{FZ1}, \cite{FZ2}. Here we consider a
particular case. Namely, let X be an affine toric surface with the
acting torus $\mathbb{T}$ and $T\subset\mathbb{T}$ be a
one-dimensional subtorus. We give complete answers to Questions
1--3 for $X$. Let $N_{T}$ and $N_{\mathbb{T}}$ be the lattices of
one parameter subgroups of tori $T$ and $\mathbb{T}$ respectively,
and $\widehat{\sigma}_{X}\subset (N_{\mathbb{T}})_{\mathbb{Q}}$ be
the cone corresponding to $X$. Up to automorphism of the lattice
$N_{\mathbb{T}}$ we may assume that the cone
$\widehat{\sigma}_{X}$ is generated by the vectors
$n_{\rho_1}=(1,0)$ and $n_{\rho_2}=(a,b)$, where $a,b\in
\mathbb{Z}_{\geqslant 0}$, $a<b$, and $\gcd(a,b)=1$.
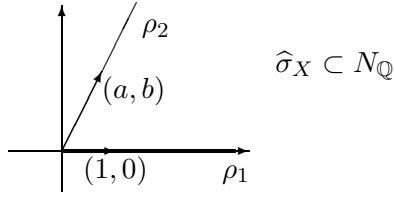
\begin{center}
\begin{picture}(100,75)
\linethickness{0.01mm} \put(30,15){\vector(1,0){90}}
\put(50,0){\vector(0,1){70}} \put(80,60){$\rho_2$}
\put(110,5){$\rho_1$} \put(130,45){$\widehat{\sigma}_{X}\subset
N_{\mathbb{Q}}$} \put(58,5){$(1,0)$} \put(65,35){$(a,b)$}
\linethickness{0.3mm} \put(50,15){\line(1,0){65}}
 \put(50,15){\line(1,2){28}}
\put(50,15){\vector(1,0){20}} \put(50,15){\vector(1,2){15}}
\end{picture}
\end{center}

The set of Demazure roots of $\widehat{\sigma}_{X}$ has the form:
$$S_{\rho_1}=\{(-1,m)\mid m\in \mathbb{Z},\,m\geqslant\frac
ab\},\,\,\, S_{\rho_2}=\{(m_1,m_2)\mid m_1,m_2\in
\mathbb{Z},\,m_1\geqslant0,\, am_1+bm_2=-1\}.$$ Let $\Gamma_{T}$
be the line corresponding to the subtorus $T$. There are three
alternatives for relative position of the cone
$\widehat{\sigma}_{X}$ and the line $\Gamma_{T}$:

\begin{enumerate}

\item $\widehat{\sigma}_{X}\cap\Gamma_{T}=\{0\}$; \item
$\widehat{\sigma}_{X}\cap\Gamma_{T}$ is a ray of
$\widehat{\sigma}_{X}$; \item $\Gamma_{T}$ intersects the interior
of $\widehat{\sigma}_{X}$.
\end{enumerate}

{\bfseries Case 1.} Suppose
$\widehat{\sigma}_{X}\cap\Gamma_{T}=\{0\}$. It follows from
Proposition \ref{sl} that
 there is only one (up to scalar) root vector for each
$\,T$-root $e$, only one $\mathbb{T}$-root has $e$ as its
restriction,  and all $\,T$-homogeneous LNDs are
$\mathbb{T}$-homogeneous as well. In the following figure we
illustrate the $T$-roots of $X$. The restrictions of the elements
of $S_{\rho_1}$ and $S_{\rho_2}$ are denoted by $"\bullet"$ and
$"\circ"$ respectively.

\vspace{0.05cm}
\begin{picture}(100,45)
\multiput(160,30)(20,0){4}{\circle*{3}}
\multiput(250,30)(10,0){5}{\circle{3}}
\put(140,30){\vector(1,0){165}}
\end{picture}
\vspace{-0.3cm}

{\bfseries Case 2.} Assume that the intersection of the cone
$\widehat{\sigma}_{X}$ with $\Gamma_{T}$ is the ray~$\rho_1$. In
this case the restriction of the set $S_{\rho_1}$ is just the
point $-1$. By Proposition \ref{sl'} the root vectors
corresponding to $e=-1$ form an infinite dimensional vector space.
For each root $e\ne-1$ there is only one $\mathbb{T}$- root, whose
restriction equals $e$. The set of the $T$-roots looks like:

\vspace{0.05cm}
\begin{picture}(100,45)
\put(170,30){\circle*{4}} \multiput(190,30)(20,0){6}{\circle{3}}
\put(140,30){\vector(1,0){165}} \put(160,18){$-1$}
\end{picture}
\vspace{-0.3cm}

{\bfseries Case 3.} Suppose $\Gamma_{T}$ intersects the interior
of the cone $\widehat{\sigma}_{X}$. Then the restriction $\pi$ on
each set $S_{\rho_i}$ is injective and the number $D=rb-qa$ is
positive for primitive lattice vector $(r,q)$ of the line
$\Gamma_{T}$. Let us prove the following lemma.

\begin{lemma} Suppose the line $\Gamma_{T}$ intersects the interior of the
cone $\widehat{\sigma}_{X}$. Then the intersection
${\pi(S_{\rho_1})\cap \pi(S_{\rho_2})}$ is not empty if and only
if $a-1$  is divisible by $\gcd(q,rb-qa)$.
\end{lemma}

\begin{proof} The restrictions of $S_{\rho_1}$ and $S_{\rho_2}$
have the form $\{-r+mq\mid m\in \mathbb{Z}_{\geqslant\frac ab}\}$
and ${\{rm_1^0+qm_2^0+kD\mid k\in \mathbb{Z}_{\geqslant0}\}},$
where $(m_1^0,m_2^0)\in S_{\rho_2}$ is a root, respectively. The
intersection of these two sets is non-empty if and only if there
exist numbers $m_0$ and $k_0$ such that
\begin{equation}\label{int} r+rm_1^0+qm_2^0=m_0q-k_0D.
\end{equation} Multiplying this equation by $a$ and using the fact
that $(m_1^0,m_2^0)\in S_{\rho_2}$, we get
\begin{equation}\label{int2} r(a-1)-Dm_2^0=m_0qa-k_0Da.
\end{equation}
So if $a-1$ is not  divisible by $\gcd(q,D)$, we obtain a
contradiction. Conversely, suppose that the left part of
\eqref{int2} is divisible by $\gcd(q,D)$. Since $b$ is divisible
by $\gcd(q,D)$ and $\gcd(a,b)=1$, the left part of \eqref{int} is
also divisible by $\gcd(q,D)$. Hence there exist $m_0$ and $k_0$
such that \eqref{int} holds.
\end{proof}

Let us draw the set of the $T$-roots as above.

\vspace{0.75cm}
\begin{picture}(100,30)
\put(70,25){$a)$} \multiput(130,30)(20,0){7}{\circle*{3}}
\multiput(140,30)(40,0){3}{\circle{3}}
\put(100,30){\vector(1,0){165}} \put(300,25){$\gcd(q,rb-qa)\nmid
a-1$}
\end{picture}

\vspace{0.06cm}
\begin{picture}(100,30)
\put(70,25){$b)$} \multiput(130,30)(20,0){7}{\circle*{3}}
\multiput(140,30)(60,0){2}{\circle{3}}
\multiput(170,30)(60,0){2}{\circle{5}}
\put(100,30){\vector(1,0){165}} \put(300,25){$\gcd(q,rb-qa)\mid
a-1$}
\end{picture}
\vspace{-0.3cm}

In case $a)$ the restriction of roots from $\mathbb{T}$ to $T$ is
bijective and any $T$-homogeneous LND is $\mathbb{T}$-homogeneous.
In case $b)$ denote by $\Lambda$ the intersection
${\pi(S_{\rho_1})\cap \pi(S_{\rho_2})}$. For the roots $e\notin
\Lambda$ we can use Theorem \ref{mt}. It remains to answer
Questions $1$ and $3$ for $e\in \Lambda$.

Now we recall a description of $T$-varieties of complexity one.
Let $N$ and $M$ be two mutually dual lattices with the pairing
denoted by $\langle,\rangle$, $\sigma$ be a cone in
$N_{\mathbb{Q}}$, $\omega\subset M_{\mathbb{Q}}$ be its dual cone,
and $C$ be a smooth curve. Consider a divisor
$\mathfrak{D}=\sum\limits_{z\in C} \Delta_{z}\cdot z,$ whose
coefficients $\Delta_{z}$ are polyhedra in $N_{\mathbb{Q}}$ with
tail cone~$\sigma$. Such objects are called {\itshape
$\sigma$-polyhedral divisors}, see~\cite{AH}. For every
${m\in\omega_{M}}$ \,\,we can get the $\mathbb{Q}$-divisor
$\mathfrak{D}(m)=\sum\limits_{z\in C} \min\limits_{p\in\Delta_z}
\langle p,m\rangle\cdot z.$ One can define the $M$-graded algebra
$${A[C,\mathfrak{D}]=\bigoplus_{m\in \omega_{M}} A_m\chi^m,}
\,\,\,\,\mbox{where}\,\,\, A_m=H^0(C,\mathfrak{D}(m))$$ and the
multiplication is determined in natural way. It follows
from~\cite{AH} that $A[C,\mathfrak{D}]$ is a normal affine domain
and that every normal affine $M$-graded domain A with $\td A = \rk
M+1$ is equivariantly isomorphic to $A[C,\mathfrak{D}]$ for some
$C$ and $\mathfrak{D}$.

Also we need a description of homogeneous LNDs of horizontal type
for $\mathbb{T}$-variety $X$ of complexity one. Below we follow
the approach given in \cite{AL}. We have
$\mathbb{K}[X]=A[C,\mathfrak{D}]$ for some $C$ and~$\mathfrak{D}$.
It turns out that $C$ is isomorphic to $\mathbb{A}^1$ or
$\mathbb{P}^1$ whenever there exists a homogeneous LND of
horizontal type on $A[C,\mathfrak{D}]$, see \cite[Lemma 3.15]{L1}.
Let $C$ be $\mathbb{A}^1$ or $\mathbb{P}^1$,
$\mathfrak{D}=\sum\limits_{z\in C}\Delta_z\cdot z$ a
$\sigma$-polyhedral divisor on $C$, $z_0\in C$, $z_{\infty}\in
C\backslash\{z_0\}$, and $v_z$ a vertex of $\Delta_z$ for every
$z\in C$. Put $C'=C$ if $C=\mathbb{A}^1$ and
$C'=C\backslash\{z_{\infty}\}$ if $C=\mathbb{P}^1$. A collection
$\widetilde{\mathfrak{D}}=\{\mathfrak{D},z_0; v_z, \forall z\in
C\}$ if $C=\mathbb{A}^1$ and
$\widetilde{\mathfrak{C}}=\{\mathfrak{C},z_0,z_{\infty}; v_z,
\forall z\in C'\}$ if $C=\mathbb{P}^1$ is called a {\itshape
colored } $\sigma$-polyhedral divisor on $C$ if the following
conditions hold:

$(1)$ $v_{\deg}:=\sum\limits_{z\in C'} v_z$ is a vertex of $\deg
\mathfrak{D}\hspace{-0.15cm}\mid_{C'}:=\sum\limits_{z\in
C'}\Delta_z$;  \hspace{0.5cm} $(2)$ $v_z\in N$ for $z\ne z_0$.

Let $\widetilde{\mathfrak{D}}$ be a colored $\sigma$-polyhedral
divisor on $C$ and $\delta\subseteq N_{\mathbb{Q}}$ be the cone
generated by ${\deg
\mathfrak{D}\hspace{-0.15cm}\mid_{C'}-v_{\deg}}$. Denote by
$\widetilde{\delta}\subseteq (N\oplus\mathbb{Z})_{\mathbb{Q}}$ the
cone generated by $(\delta,0)$ and $(v_{z_0},1)$ if
$C=\mathbb{A}^1$, and by $(\delta,0)$, $(v_{z_0},1)$ and
$(\Delta_{z_{\infty}}+v_{\deg}-v_{z_0}+\delta,-1)$ if
$C=\mathbb{P}^1$. By definition, put $d$ the minimal positive
integer such that $d\cdot v_{z_0}\in N$. A pair
$(\widetilde{\mathfrak{D}},e)$, where $e\in M$, is said to be
{\itshape coherent } if

$(1)$ There exists $s\in \mathbb{Z}$ such that
$\widetilde{e}=(e,s)\in M\oplus \mathbb{Z}$ is a Demazure root of
the cone $\widetilde{\delta}$ with distinguished ray
$\widetilde{\rho}=(d\cdot v_{z_0},d)$.

$(2)$ $v(e)\geqslant 1+v_z(e)$ for every $z\in C'\backslash
\{z_0\}$ and every vertex $v\ne v_{z}$ of the polyhedron
$\Delta_{z}$.

$(3)$ $d\cdot v(e)\geqslant 1+v_{z_0}(e)$ for every vertex $v\ne
v_{z_0}$ of the polyhedron $\Delta_{z_0}$.

$(4)$ If $Y=\mathbb{P}^1$, then $d\cdot v(e)\geqslant -1-d\cdot
\sum\limits_{z\in Y'} v_{z}(e)$ for every vertex $v$ of the
polyhedron $\Delta_{z_{\infty}}$.

It follows from \cite[Theorem 1.10]{AL} that homogeneous LNDs of
horizontal type on $A[C,\mathfrak{D}]$ are in bijection with the
coherent pairs $(\widetilde{\mathfrak{D}},e)$.

Let us return to our case. Following \cite[Section 11]{AH} we will
show how to determine $C$ and $\mathfrak{D}$ such that
${X\stackrel{T}{\cong}}\spec A[C,\mathfrak{D}]$. Since
$T\subset\mathbb{T}$, we have an exact sequence
$$ \begin{CD}
0 @>>> N_{T} @>F>> N_{\mathbb{T}} @>P>> N_{\mathbb{T}}/N_{T} @>>>
0. \end{CD}$$ Let $\Sigma$ be the coarsest quasifan in
$(N_{\mathbb{T}}/N_{T})_{\mathbb{Q}}$ refining all cones
$P(\tau)$, where $\tau$ runs over all faces
of~$\widehat{\sigma}_{X}$. It follows from \cite[Section~6]{AH}
that the desired curve $C$ is the toric variety corresponding
to~$\Sigma$. Let us choose the projection $s:
N_{\mathbb{T}}\rightarrow N_{T}$, which satisfies $s\circ F=id$.
The divisor $\mathfrak{D}$ on $C$ such that
$X\stackrel{T}{\cong}\spec A[Y,\mathfrak{D}]$ is $
\mathfrak{D}=\sum_{\rho\in \Sigma(1)} \Delta_{\rho}\cdot
D_{\rho},$ where $\Sigma(1)\subset\Sigma$ is the set of
one-dimensional cones, $D_{\rho}$ is the prime divisor
corresponding to $\rho$, $\Delta_{\rho}=s(\widehat{\sigma}_{X}\cap
P^{-1}(n_{\rho}))\subset (N_{T})_{\mathbb{Q}}$, and $n_{\rho}$ is
the primitive lattice vector on $\rho$. Here all $\Delta_{\rho}$
are $\sigma$-tailed polyhedra with
$\sigma=s(\widehat{\sigma}_{X}\cap (F(N_{T}))_{\mathbb{Q}}$. In
our case

$$F= \binom rq
 \,\, ,\,\, s=(u,v) \,\,\,\mbox{and}\,\,\,  P=( q,-r),$$ where
 $u,v \in \mathbb{Z}$, $ru+qv=1$. By direct calculation we
 obtain $Y=\mathbb{P}^1$, $\sigma=\mathbb{Q}_{\geqslant0}$, and
 $$\mathfrak{D}=\left(p_1+\sigma\right)\cdot
[0]+\left(p_2+\sigma\right)\cdot [\infty],$$ where $p_1=\dfrac uq$
and $p_2=\dfrac{au+bv}{rb-qa}$. In this case there is no
homogeneous LND of fiber type on $A[C,\mathfrak{D}]$. Homogeneous
LNDs of horizontal type  of degree $e$ are in bijection with the
coherent pairs $(\widetilde{\mathfrak{D}},e)$. If neither $p_1$
nor $p_2$ belong to $\mathbb{Z}$, then we should put either
$z_0=0$ and $z_{\infty}=\infty$, or $z_0=\infty$ and
$z_{\infty}=0$ in the definition of colored $\sigma$-polyhedral
divisor $\widetilde{\mathfrak{D}}$. It means that there exist at
most two LNDs of degree $e$. So for a root $e\in \Lambda$ there
are two root vectors and they are $\mathbb{T}$-homogeneous. It is
clear that $p_1$ is integer if and only if $q$ equals 1. Let us
find out when $p_2$ is integer.

\begin{lemma} The number $p_2=\dfrac{au+bv}{rb-qa}$ is integer if and
only if $rb-qa=1$.
\end{lemma}
\begin{proof} Denote $D=rb-qa$. Suppose $au+bv=kD$, where $k\in\mathbb{Z}$. We also
know that $ru+qv=1$. Using these equations, we get
$$a=aru+aqv=D(rk-v),\,\,\,\, b=bru+bqv=D(qk-u).$$ So $a$ and $b$
are divisible by $D$. This holds if and only if $D=1$.
\end{proof}

Now suppose $p_1\notin \mathbb{Z}$ and $p_2\in\mathbb{Z}$. Then
for any colored divisor either $z_0=0$ or $z_{\infty}=0$ holds. It
is not hard to check that in the former case the pair
$(\widetilde{\mathfrak{D}},e)$ is coherent if and only if ${\frac
1q+ep_1\in \mathbb{Z}}$ and $qep_2+1\geqslant0$. These conditions
do not depend on the choice of $z_{\infty}$. So if there is one
coherent pair $(\widetilde{\mathfrak{D}},e)$ with $z_0=0$, then
any pair $(\widetilde{\mathfrak{D}}',e)$ with $z_0=0$ and
$z_{\infty}\in \mathbb{P}^1\setminus\{z_0\}$ is coherent. Since
$p_2\in\mathbb{Z}$, we consider all $\widetilde{\mathfrak{D}}$
with $z_{\infty}=0$ as the same colored $\sigma$-polyhedral
divisor. Let $e\in \Lambda$. We already know two root vectors of
degree $e$, which are $\mathbb{T}$-homogeneous LNDs, whose degrees
belong to different $S_{\rho_i}$. Hence the inequality $z_0\ne
z_0'$ holds for corresponding coherent pairs
$(\widetilde{\mathfrak{D}},e)$ and
$(\widetilde{\mathfrak{D}}',e)$. Thus there are as many root
vectors corresponding to $e$ as points on $\mathbb{P}^1$. All
these root vectors are described by Proposition \ref{nhrv}. The
cases $p_1\in \mathbb{Z}$, $p_2\notin\mathbb{Z}$ and
$p_1,p_2\in\mathbb{Z}$ are analyzed similarly.

The following figure illustrates the lines $\Gamma_{T}$ with
$p_1\in \mathbb{Z}$ or $p_2\in \mathbb{Z}$.

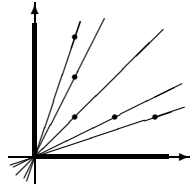
\vspace{0.6cm}
\begin{center}
\begin{picture}(100,75)
\multiput(65,45)(15,0){3}{\circle*{2}}
\multiput(65,60)(0,15){2}{\circle*{2}}
\put(50,20){\vector(0,1){68}} \put(40,30){\vector(1,0){68}}
\put(43,23){\line(1,1){55}} \put(46,22){\line(1,2){30}}
\put(47,21){\line(1,3){20}} \put(42,26){\line(2,1){63}}
\put(41,27){\line(3,1){65}}
\linethickness{0.5mm} \put(50,30){\line(1,0){50}}
\put(50,30){\line(0,1){50}}
\end{picture}
\end{center}
\vspace{-0.5cm}

Summarizing, we obtain the following statement.

\begin{proposition} Let X be an affine toric surface with the acting torus
$\mathbb{T}$ and $T\subset\mathbb{T}$ be a one-dimensional
subtorus. Suppose that the cone $\widehat{\sigma}_{X}$
corresponding to $X$ is generated by $n_{\rho_1}=(1,0)$ and
$n_{\rho_2}=(a,b)$, where $a,b\in \mathbb{Z}_{\geqslant 0}$,
$a<b$, and $\gcd(a,b)=1$, and the line $\Gamma_{T}$ is generated
by $(r,q)$, where $\gcd(r,q)=1$. Let $\pi$ be the restriction of
roots from $\mathbb{T}$ to $T$, $e$ be a $T$-root,
$\Lambda={\pi(S_{\rho_1})\cap \pi(S_{\rho_2})}$, and $D=rb-qa$.
Then the following table describes the restriction of roots.
\begin{center}
\begin{tabular}{||l|c|c|c||}
\hline &The number of&The number of $\mathbb{T}$-roots &Are all
$T$-homogeneous\\
&root vectors&having $e$ as their &LNDs on $X$\\
&corresponding to $e$&restriction& $\mathbb{T}$-homogeneous?\\
\hline\hline
{\bfseries 1.}  $\widehat{\sigma}_{X}\cap \Gamma_{T}=\{0\}$&1&1& yes\\
\hline\hline {\bfseries 2.} $\widehat{\sigma}_{X}\cap
\Gamma_{T}=\mathbb{Q}_{\geqslant0}\cdot\rho_i$&&&\\
\hspace{1.2cm}$e=-1$&an infinite&$\infty$&no\\
&dimensional space&&\\
\hspace{1.2cm}$e\ne-1$&1&1&yes\\
\hline\hline {\bfseries 3.} $\widehat{\sigma}_{X}^{\circ}\cap \Gamma_{T}\ne\varnothing$&&&\\
{\bfseries 3.1.} $\gcd(q,D)\nmid a-1$$$&1&1& yes\\
\hline
{\bfseries 3.2.} $q\ne1, D\ne1$ and&&&\\
\hspace{0.5cm}$\gcd(q,D)\mid a-1$&&&\\
\hspace{1.2cm}$e\notin\Lambda$&1&$1$&yes\\
\hspace{1.2cm}$e\in\Lambda$&2&2&yes\\
\hline
{\bfseries 3.3.} $q=1$ or $D=1$&&&\\
\hspace{1.2cm}$e\notin\Lambda$&1&$1$&yes\\
\hspace{1.2cm}$e\in\Lambda$&as many as points&2&no\\
&on $\mathbb{P}^1$&&\\\hline
\end{tabular}
\end{center}
\end{proposition}

\section*{Acknowledgements}
The author is grateful to I.V. Arzhantsev for posing the problem
and permanent support. Thanks are also due to A. Liendo for useful
comments.

\end{document}